\documentclass[12pt, 14paper,reqno]{amsart}
\usepackage{amsmath,amsfonts,amssymb}
\usepackage{longtable}
\theoremstyle{plain}
\newtheorem{theorem}{Theorem}

\newtheorem{corollary}{Corollary}
\newtheorem{proposition}{Proposition}
\newtheorem{conj}{Conjecture}
\theoremstyle{proof}
\theoremstyle{definition}

\theoremstyle{remark}

\theoremstyle{lamma}

\numberwithin{equation}{section}
\usepackage{amsmath}
\usepackage{amsfonts}   
\usepackage{amssymb}
\usepackage{amssymb, amsmath, amsthm}
\usepackage[breaklinks]{hyperref}

\theoremstyle{thmrm}

\newenvironment{dedication}
    {\vspace{3ex}\begin{quotation}\begin{center}\begin{em}}
    {\par\end{em}\end{center}\end{quotation}}
    
\begin{document}
\title{Exponent of class group of certain imaginary quadratic fields}
\author{Kalyan Chakraborty and Azizul Hoque}
\address{Kalyan Chakraborty @Kalyan Chakraborty, Harish-Chandra Research Institute, HBNI,
Chhatnag Road, Jhunsi, Allahabad 211 019, India.}
\email{kalyan@hri.res.in}

\address{Azizul Hoque @Azizul Hoque Harish-Chandra Research Institute, HBNI,
Chhatnag Road, Jhunsi,  Allahabad 211 019, India.}
\email{ azizulhoque@hri.res.in}

\maketitle
\begin{dedication}
To Prof. Wenpeng Zhang on his $60^{th}$ birthday with great respect and friendship
\end{dedication}
\begin{abstract}
Let $n>1$ be an odd integer. We prove that there are infinitely many imaginary quadratic fields of the form $\mathbb{Q}(\sqrt{x^2-2y^n})$ whose ideal class group has an element of order $n$. This family gives a counter example to a conjecture by H. Wada \cite{WA70} on the structure of ideal class groups.
\end{abstract}

\section{Introduction}
Let $x, y$ and $n$ be positive integers. We consider the family of imaginary quadratic fields
$$
K_{x,y,n, \mu}=\mathbb{Q}(\sqrt{x^2-\mu y^n})
$$ 
with the conditions: $\gcd (x, y)=1,~ y>1,~ \mu\in \{1, 2, 4\}$ and $x^2\leq \mu y^n$.
Let $\mathcal{H}(K_{x,y,n,\mu})$ and $\mathcal{C}(K_{x,y,n,\mu})$ be respectively denote the class number and (ideal) class group of $K_{x, y,n, \mu}$. For $\mu \in \{1, 4\}$, there are many results concerning the divisibility of $\mathcal{H}(K_{x,y,n,\mu})$. We highlights some important results for these values of $\mu$ in the next two paragraphs. 

In 1922, T. Nagell \cite{NA22} proved that $\mathcal{H}(K_{x,y,n,1})$ is divisible by $n$ if both $n$ and $y$ are odd, and $\ell\mid x$, but $\ell^2 \nmid x$ for all prime divisors $\ell$ of $n$. Let $s$ be the square factor of $x^2-y^n$, that is
$$
x^2-y^n=-s^2D,
$$
where $D>0$ is the square-free part of $x^2-y^n$. For $s=1$, N. C. Ankeny and S. Chowla \cite{AC55} proved that $\mathcal{H}(K_{x,3,n,1})$ is divisible by $n$ if both $n$ and $x$ are even, and $x< (2\times 3^{n-1})^{1/2}$. In 1998, M. R. Murty \cite{RM98} considered the divisibility of $\mathcal{H}(K_{1,y,n,1})$ by $n$, when $s=1$ and $n\geq 5$ is odd. In the same paper, he further discussed this result when $s<\frac{y^{n/4}}{2^{3/2}}$. He also discussed a more general case, that is the divisibility of $\mathcal{H}(K_{x,y,n,1})$ in \cite{RM99}. K. Soundararajan \cite{SO00} (resp. A. Ito \cite{IT11}) studied the divisibility of $\mathcal{H}(K_{x,y,n,1})$ by $n$ under the condition $s<\sqrt{(y^n-x^2)/(y^{n/2}-1)}$ (resp. each prime divisor of $s$ divides $D$ also). Furthermore, Y. Kishi \cite{KI09} (resp. A. Ito \cite{IT11}, and M. Zhu and T. Wang \cite{MW12}) studied the divisibility by $n$ of $\mathcal{H}(K_{2^k,3,n,1})$ (resp. $\mathcal{H}(K_{2^k,p,n,1})$ with $p$ odd prime and $\mathcal{H}(K_{2^k,y,n,1})$ with $y$ odd integer). In \cite{HS16}, Hoque and Saikia investigated the divisibility of $\mathcal{H}(K_{x,p, n,1})$ for any positive integer $x$, odd integer $n$ and any prime $p$ satisfying certain assumptions. Recently, K. Chakraborty et al. \cite{CHKP18} discussed the divisibility by $n$ of $\mathcal{H}(K_{p,q,n,1})$ when both $p$ and $q$ are odd primes, and $n$ is an odd integer.    

On the other hand, B. H. Gross and D. E. Rohrlich \cite{GR01} (resp. J. H. E. Cohn \cite{CO02}, and K. Ishii \cite{IS11}) proved the divisibility by $n$ of $\mathcal{H}(K_{1,y,n,4})$ (resp. $\mathcal{H}(K_{1,2,n,4})$ except for $n=4$, and $\mathcal{H}(K_{1,y,n,4})$ for even $n$) for an odd integer $n$. Further, S. R. Louboutin \cite{LO09} proved that $\mathcal{C}(K_{1,y,n,4})$ has an element of order $n$ if atleast one odd prime divisor of $y$ is equal to $3\pmod 4$. 
Recently, A. Ito \cite{IT15} discussed the divisibility of $\mathcal{H}(K_{3^e,y,n,4})$ by $n$ under certain conditions. 

More recently, K. Chakraborty and A. Hoque \cite{HC18} (resp. in\cite{CH18}) proved the divisibility by $3$ of $\mathcal{H}(K_{1,y,3,2})$ (resp. $\mathcal{H}(K_{\ell, k^n,3,2\ell})$ for certain odd integers $\ell, n, k$) for any odd integer $y$. One can consult the survey \cite{CHS} for more information about the divisibility of class number of quadratic fields. In this paper, we show that $\mathcal{C}(K_{p,q,n, 2})$ has an element of order $n$ when both $p$ and $q$ are odd primes, and $n$ is an odd integer. Namely, we prove:

\begin{theorem}\label{theorem1}
Let $p$ and $q$ be distinct odd primes, and let $n\geq 3$ be an odd integer with $p^2<2q^n$ and $2q^n-p^2\ne \square$. Assume that $3q^{n/3}\ne p+2$ whenever $3\mid n$. Then $\mathcal{C}(K_{p, q,n, 2})$ has an element of order $n$.
\end{theorem}
An immediate consequence of the above result is:
\begin{corollary}\label{cor}
Let $p, q$ and $n$ as in Theorem \ref{theorem1}. Then there are infinitely many imaginary quadratic fields with discriminants of the form $p^2-2q^n$ whose class number is divisible by $n$.
\end{corollary}
The present family of imaginary quadratic fields provides a counter example of a conjecture (namely, Conjecture \ref{conjw}) given by H. Wada \cite{WA70} in 1970. We define some notations before stating this conjecture. 
The class group of a number field can be expressed, by the structure theorem of abelian groups, as the direct product of cyclic groups of orders $h_1, h_2, \cdots, h_t$. We denote the direct product $C_{h_1}\times C_{h_2}\times \cdots\times C_{h_t}$ of cyclic groups by $[h_1, h_2, \cdots, h_t]$. By Gauss's genus theory, if there are $r$ number of distinct rational primes that ramifies in $\mathbb{Q}(\sqrt{-D})$, for some square-free integer $D>1$, then the $2$-rank of its class group is $r-1$. In other words, the $2$-Sylow subgroup of its class group has rank $r-1$. It is noted that the $2$-Sylow subgroup of class group tends to $r-2$ elementary $2$-groups and one large cyclic factor collecting the other powers of $2$ in the class number so that the $2$-Sylow subgroup of the subgroup of squares is cyclic. On the other hand, $r-1$ number of even integers are there among $h_1, h_2, \cdots, h_t$. Sometimes, the structure of the class group of $\mathbb{Q}(\sqrt{-D})$ can be trivially determined by $r$ and the class number, $h=h_1h_2\cdots h_t$. In this case, the group is cyclic when $r=1$ or $r=2$ or it is of the type $(h_1, h_2, 2^{r_1}, \cdots,2^{r_k})$ when $r\geq 3$. In this context, H. Wada \cite{WA70} gave the following conjecture.
\begin{conj}[Wada \cite{WA70}]\label{conjw}
All the class groups of imaginary quadratic fields are either cyclic or of the type $(h_1, h_2, 2^{r_1}, 2^{r_2}, \cdots, 2^{r_k})$.  
\end{conj} 
In Table \ref{Table2}, we find a class group of the type $(h_1, h_2, h_3, 2^{r_1}, 2^{r_2}, \cdots, 2^{r_k})$ (see $**$ mark) which is not cyclic. This gives a counter example to Conjecture \ref{conjw}.

\section{Proof of Theorem \ref{theorem1}}
We begin the proof with the following crucial proposition.

\begin{proposition}\label{prop}
Let $p, q$ and $n$ as in Theorem \ref{theorem1}, and let $s$ be the positive integer such that 
\begin{equation}\label{eq1}
p^2-2q^n=-s^2D,
\end{equation}
where $D$ is a square-free positive integer. Then for $\alpha=p+s\sqrt{-D}$, and for any prime divisor $\ell$ of $n$, $2\alpha$ is not an $\ell$-th power of an element in the ring of integers of $K_{p,q,n, \mu}$.
\end{proposition}

\begin{proof}
Let $\ell$ be a prime such that $\ell\mid n$. Then $\ell$ is odd since $n$ is odd.
From \eqref{eq1}, we see that $-D\equiv 3 \pmod 4$, and thus if $2\alpha$ is an $\ell$-th power of an element in the ring of integers $\mathcal{O}_{K_{p,q,n,2}}$ in $K_{p,q,n,2}$, then we can write
\begin{equation}\label{eq2}
2\alpha=(u+v\sqrt{-D})^{\ell}
\end{equation}
for some $u,v\in \mathcal{O}_{K_{p,q,n,2}}$. Taking norm, we obtain
\begin{equation}\label{eq3}
8q^n=(u^2+Dv^2)^{\ell}.
\end{equation}
This shows that $\ell=3$ with $3\mid n$. Thus \eqref{eq3} reduces to 
\begin{equation}\label{eq4}
2q^m=u^2+Dv^2,
\end{equation}
 where $n=3m$ for some integer $m>0$. Now \eqref{eq4} implies that $u$ and $v$ either both odd or both even since $D$ is odd. If both $u$ and $v$ are even, then $2\mid q$ which is a contradiction.
It remains to treat the case when both $u$ and $v$ are odd. We compare the real and imaginary parts on both sides of \eqref{eq2}, and get
\begin{equation}\label{eq5}
2p=u^3-3uv^2D
\end{equation}
and 
\begin{equation}\label{eq6}
 2s=3u^2v-v^3D.
\end{equation}  
We see that \eqref{eq5} implies $u\mid 2p$. As $u$ is an odd integer and $p$ is odd prime, we must have
$u=\pm 1$ or $u=\pm p$.

If $u=1$, then \eqref{eq5} would imply $2p=1-3v^2D$, which is a contradiction. Similarly, if $u=p$ then \eqref{eq5} gives 
\begin{equation}\label{eq7}
2=p^2-3v^2D.
\end{equation} 
Reading \eqref{eq7} modulo $3$, we see that 
$p^2\equiv 2 \pmod 3$ which again leads to a contradiction.

Again when $u=-1$, the relations \eqref{eq5} and \eqref{eq4} give, 
\begin{equation}\label{eq8}
3Dv^2-1=2p
\end{equation}
and 
\begin{equation}\label{eq9}
Dv^2+1=2q^m.
\end{equation}
We now add \eqref{eq8} and \eqref{eq9}, and that  gives
\begin{equation}\label{eq10}
2Dv^2=p+q^m.
\end{equation}
Now subtracting \eqref{eq9} from \eqref{eq8} which leads us to
\begin{equation}\label{eq11}
Dv^2=1+p-q^m.
\end{equation}
In this case finally \eqref{eq10} and \eqref{eq11} together give
$3q^m=p+2$. This contradicts the assumption.

We are now left with the case when $u=-p$. In this case \eqref{eq4}, \eqref{eq5} and \eqref{eq6} become, 
\begin{equation}\label{eq12}
Dv^2+p^2=2q^m,
\end{equation}

\begin{equation}\label{eq13}
3Dv^2-p^2=2
\end{equation}
 and 
 \begin{equation}\label{eq14}
Dv^3+2s=3p^2v. 
 \end{equation}
Multiplying \eqref{eq12} by $3$, and then subtracting the resulting equation from \eqref{eq13}, we get
\begin{equation}\label{eq15}
2p^2+1=3q^m.
\end{equation}  
Also \eqref{eq14} shows that $v\mid s$, and thus we can write $s=vt$ for some odd integer $t$, since both $v$ and $s$ are odd. Therefore \eqref{eq14} can be rewritten as  
\begin{equation}\label{eq16}
Dv^2+2t=3p^2.
\end{equation} 
 Since $s=vt$, so that \eqref{eq1} becomes
 \begin{equation*}
 Dv^2t^2+p^2=2q^n.
 \end{equation*}
 Using \eqref{eq16}, we obtain
 \begin{equation*}
 3p^2t^2-2t^3+p^2=2q^n.
 \end{equation*}
 Since $n=3m$, so that using \eqref{eq15} the above equation reduces to
 \begin{equation*}
 81p^2t^2-54t^3+27t^2=2\left(2p^2+1\right)^3.
 \end{equation*}
 Finally reading this equation modulo $4$, we arrive at
$$ 0\equiv 2\pmod 4,$$
which is not possible. This completes the proof.
\end{proof}

\subsection*{Proof of Theorem \ref{theorem1}}\hspace*{8cm}\\
Let $s$ be the positive integer such that 
$$
p^2-2q^n=-s^2D,
$$
where $D$ is a square-free positive integer.
Let $\alpha:=p+s\sqrt{-D}$. Then $$N(\alpha)=\alpha\bar{\alpha}=2q^n,$$
where $\bar{\alpha}$ is the conjugate of $\alpha$, and they are coprime. 

We see that $q$ splits completely in $K_{p,q,n, 2}=\mathbb{Q}(\sqrt{-D})$, and thus we have,
$$
(q)=\mathfrak{q}\mathfrak{q}',
$$
where $\mathfrak{q}$ and its conjugate $\mathfrak{q}'$ are prime ideals in $\mathcal{O}_{K_{p,q,n,2}}$ different from each other.
Moreover, $(2)=\mathfrak{p}^2$ with $\mathfrak{p}=(2,1+\sqrt{-D})$. 

We see that $(\alpha)$ is not divisible by any rational integer other than $\pm 1$, and thus we can consider the following decomposition of $(\alpha)$:
$$
(\alpha)=\mathfrak{p}\mathfrak{q}^m.
$$
Then $N(\alpha)=2q^m$ and hence $n=m$. 

We now put $\mathfrak{A}= \mathfrak{p}\mathfrak{q}$. Then (since $n$ is odd)
$$\mathfrak{A}^n=\mathfrak{p}^{n-1}(\mathfrak{p}\mathfrak{q}^n)=(2^{\frac{n-1}{2}}\alpha)\subseteq(2\alpha),$$
which is principal in $\mathcal{O}_{K_{p,q,n,2}}$. Thus if $[\mathfrak{A}]$ denotes the ideal class containing $\mathfrak{A}$ then by Proposition \ref{prop}, we see that the order of $[\mathfrak{A}]$ is $n$. This completes the proof. 
\section{Numerical Examples}
Here, we provide some numerical values to corroborate Theorem \ref{theorem1}. All the computations in this paper were done using PARI/GP (version 2.7.6) \cite{PARI}. Table \ref{Table} gives the list of imaginary quadratic fields $K_{p,q,n,2}$ corresponding to the distinct primes $p$ and $q$ not larger than $17$, and odd integer $3\leq n\leq 19$. We see that absolute discriminants are not exceeding $5\times 10^{23}$, and the corresponding class numbers are very large which can go upto (about) $5.5\times 10^{11}$. It is noted that this list does not exhaust all the imaginary quadratic fields $K_{p,q,n,2}$ of discriminants not exceeding $5\times 10^{23}$. In the Table \ref{Table}, we use $*$ mark in the column for class number to indicate the failure of the assumption $``3q^{n/3}\ne p+2"$ of Theorem \ref{theorem1}.

\begin{center}
{\tiny
\begin{longtable}{|l l l   l l l l l l l|}
\caption{Numerical examples of Theorem 1.} \label{Table} \\

\hline \multicolumn{1}{|c}{$n$} & \multicolumn{1}{c}{$p$} & \multicolumn{1}{c}{$q$}& \multicolumn{1}{c}{$p^2-2q^n$} & \multicolumn{1}{c}{$h(-D)$}& \multicolumn{1}{c}{$n$} & \multicolumn{1}{c}{$p$} & \multicolumn{1}{c}{$q$}& \multicolumn{1}{c}{$p^2-2q^n$} & \multicolumn{1}{c|}{$h(-D)$}\\ \hline 
\endfirsthead

\multicolumn{10}{c}%
{{\bfseries \tablename\ \thetable{} -- continued from previous page}} \\
\hline \multicolumn{1}{|c}{$n$} & \multicolumn{1}{c}{$p$} & \multicolumn{1}{c}{$q$}& \multicolumn{1}{c}{$p^2-2q^n$} & \multicolumn{1}{c}{$h(-D)$}& \multicolumn{1}{c}{$n$} & \multicolumn{1}{c}{$p$} & \multicolumn{1}{c}{$q$}& \multicolumn{1}{c}{$p^2-2q^n$} & \multicolumn{1}{c|}{$h(-D)$}\\ \hline 
\endhead
\hline \multicolumn{10}{|r|}{{Continued on next page}} \\ \hline
\endfoot
\hline 
\endlastfoot
 3 & 3  & 5 & -241 & 12 &   3 & 3  & 7 & -677 & 30\\ \hline
 3 & 3  & 11 & -2653 & 24 & 3 & 3  & 13 & -4385 & 96\\ \hline
 3 & 3  & 17 & -9817 & 48 & 3 & 5  & 3 & -29 & 6\\ \hline
 3 & 5  & 7 & -661 & 18 &  3 & 5  & 11 & -2637 & 36\\ \hline
 3 & 5  & 13 & -4369 & 48 &  3 & 5  & 17 & -9801 & 72 \\ \hline
 3 & 7  & 3 & -5 & 2* & 3 & 7  & 5 & -201 & 12\\ \hline
 3 & 7  & 11 & -2613 & 24 &  3 & 7  & 13 & -4345 & 48 \\ \hline
 3 & 7  & 17 & -9777 & 60 &  3 & 11  & 5 & -129 & 12\\ \hline
 3 & 11  & 7 & -565 & 12 &  3 & 11  & 13 & -4273 & 24\\ \hline
 3 & 11  & 17 & -9705 & 72 &  3 & 13  & 5 & -81 & 1*\\ \hline
 3 & 13  & 7 & -517 & 12 &  3 & 13  & 11 & -2493 & 24\\ \hline
 3 & 13  & 17 & -9657 & 48 &  3 & 17  & 7 & -397 & 6 \\ \hline
 3 & 17  & 11 & -2373 & 24 &  3 & 17  & 13 & -4105 & 48\\ \hline
 5 & 3  & 5 & -6241 & 40 &  5 & 3  & 7 & -33605 & 240 \\ \hline
 5 & 3  & 11 & -322093 & 150 &  5 & 3  & 13 & -742577 & 800 \\ \hline
 5 & 3  & 17 & -2839705 & 800 &  5 & 5  & 3 & -461 & 30\\ \hline
 5 & 5  & 7 & -33589 & 150 &  5 & 5  & 11 & -322077 & 280\\ \hline
 5 & 5  & 13 & -742561 & 500 &  5 & 5  & 17 & -2839689 & 1760\\ \hline
 5 & 7  & 3 & -437 & 20 &  5 & 7  & 5 & -6201 & 80\\ \hline
 5 & 7  & 11 & -322053 & 320 &  5 & 7  & 13 & -742537 & 380\\ \hline
 5 & 7  & 17 & -2839665 & 1120 &  5 & 11  & 3 & -365 & 20\\ \hline
 5 & 11  & 5 & -6129 & 60 &  5 & 11  & 7 & -33493 & 70\\ \hline
 5 & 11  & 13 & -742465 & 480 &  5 & 11  & 17 & -2839593 & 800\\ \hline
 5 & 13  & 3 & -317 & 10 &  5 & 13  & 5 & -6081 & 60 \\ \hline
 5 & 13  & 7 & -33445 & 80 &  5 & 13  & 11 & -321933 & 440\\ \hline
 5 & 13  & 17 & -2839545 & 960 &  5 & 17  & 3 & -197 & 10 \\ \hline
 5 & 17  & 5 & -5961 & 60 &  5 & 17  & 7 & -33325 & 120 \\ \hline
 5 & 17  & 11 & -321813 & 240 &  5 & 17  & 13 & -742297 & 480 \\ \hline
 7 & 3  & 5 & -156241 & 168 &  7 & 3  & 7 & -1647077 & 1260 \\ \hline
 7 & 3  & 11 & -38974333 & 2926 &  7 & 3  & 13 & -125497025 & 11648 \\ \hline
 7 & 3  & 17 & -820677337 & 10724 &  7 & 5  & 3 & -4349 & 42 \\ \hline
 7 & 5  & 7 & -1647061 & 896 &  7 & 5  & 11 & -38974317 & 2968 \\ \hline
 7 & 5  & 13 & -125497009 & 5824 &  7 & 5  & 17 & -820677321 & 26432 \\ \hline
 7 & 7  & 3 & -4325 & 84 &  7 & 7  & 5 & -156201 & 308 \\ \hline
 7 & 7  & 11 & -38974293 & 3696 &  7 & 7  & 13 & -125496985 & 5768 \\ \hline
 7 & 7  & 17 & -820677297 & 19096 &  7 & 11  & 3 & -4253 & 42 \\ \hline
 7 & 11  & 5 & -156129 & 392 &  7 & 11  & 7 & -1646965 & 588 \\ \hline
 7 & 11  & 13 & -125496913 & 4704 &  7 & 11  & 17 & -820677225 & 18144 \\ \hline
 7 & 13  & 3 & -4205 & 56 &  7 & 13  & 5 & -156081 & 364 \\ \hline
 7 & 13  & 7 & -1646917 & 728 &  7 & 13  & 11 & -38974173 & 3360 \\ \hline
 7 & 13  & 17 & -820677177 & 16800 &  7 & 17  & 3 & -4085 & 56 \\ \hline
 7 & 17  & 5 & -155961 & 224 &  7 & 17  & 7 & -1646797 & 658 \\ \hline
 7 & 17  & 11 & -38974053 & 3780 &  7 & 17  & 13 & -125496745 & 7392 \\ \hline
 9 & 3  & 5 & -3906241 & 1440 &  9 & 3  & 7 & -80707205 & 11448 \\ \hline
 9 & 3  & 11 & -4715895373 & 29556 &  9 & 3  & 13 & -21208998737 & 162432 \\ \hline
 9 & 3  & 17 & -237175752985 & 337176 &  9 & 5  & 3 & -39341 & 198 \\ \hline
 9 & 5  & 7 & -80707189 & 7272 &  9 & 5  & 11 & -4715895357 & 67716 \\ \hline
 9 & 5  & 13 & -21208998721 & 62136 &  9 & 5  & 17 & -237175752969 & 349164 \\ \hline
 9 & 7  & 3 & -39317 & 162 &  9 & 7  & 5 & -3906201 & 2448 \\ \hline
 9 & 7  & 11 & -4715895333 & 28512 &  9 & 7  & 13 & -21208998697 & 57744 \\ \hline
 9 & 7  & 17 & -237175752945 & 463824 &  9 & 11  & 3 & -39245 & 288 \\ \hline
 9 & 11  & 5 & -3906129 & 1692 &  9 & 11  & 7 & -80707093 & 3852 \\ \hline
 9 & 11  & 13 & -21208998625 & 79200 &  9 & 11  & 17 & -237175752873 & 284256 \\ \hline
 9 & 13  & 3 & -39197 & 108 &  9 & 13  & 5 & -3906081 & 1512 \\ \hline
 9 & 13  & 7 & -80707045 & 4608 &  9 & 13  & 11 & -4715895213 & 33300 \\ \hline
 9 & 13  & 17 & -237175752825 & 228096 &  9 & 17  & 3 & -39077 & 180 \\ \hline
 9 & 17  & 5 & -3905961 & 1368 &  9 & 17  & 7 & -80706925 & 5184 \\ \hline
 9 & 17  & 11 & -4715895093 & 35712 &  9 & 17  & 13 & -21208998457 & 74376 \\ \hline
 11 & 3  & 5 & -97656241 & 3608 &  11 & 3  & 7 & -3954653477 & 46332 \\ \hline
 11 & 3  & 11 & -570623341213 & 286770 &  11 & 3  & 13 & -3584320788065 & 2956800 \\ \hline
 11 & 3  & 17 & -68543792615257 & 2056120 &  11 & 5  & 3 & -354269 & 704 \\ \hline
 11 & 5  & 7 & -3954653461 & 36432 &  11 & 5  & 11 & -570623341197 & 519200 \\ \hline
 11 & 5  & 13 & -3584320788049 & 875072 &  11 & 5  & 17 & -68543792615241 & 6392760 \\ \hline
 11 & 7  & 3 & -354245 & 528 &  11 & 7  & 5 & -97656201 & 6864 \\ \hline
 11 & 7  & 11 & -570623341173 & 340032 &  11 & 7  & 13 & -3584320788025 & 1146464 \\ \hline
 11 & 7  & 17 & -68543792615217 & 5876112 &  11 & 11  & 3 & -354173 & 528 \\ \hline
 11 & 11  & 5 & -97656129 & 8712 &  11 & 11  & 7 & -3954653365 & 39776 \\ \hline
 11 & 11  & 13 & -3584320787953 & 797720 &  11 & 11  & 17 & -68543792615145 & 3800544 \\ \hline
 11 & 13  & 3 & -354125 & 660 &  11 & 13  & 5 & -97656081 & 9944 \\ \hline
 11 & 13  & 7 & -3954653317 & 37268 &  11 & 13  & 11 & -570623341053 & 570240 \\ \hline
 11 & 13  & 17 & -68543792615097 & 4511232 &  11 & 17  & 3 & -354005 & 352 \\ \hline
 11 & 17  & 5 & -97655961 & 7216 & 11 & 17  & 7 & -3954653197 & 25872 \\ \hline
 11 & 17  & 11 & -570623340933 & 353760 &  11 & 17  & 13 & -3584320787785 & 1076416 \\ \hline
 13 & 3  & 5 & -2441406241 & 29432 &  13 & 3  & 7 & -193778020805 & 435968 \\ \hline
 13 & 3  & 11 & -69045424287853 & 4435704 &  13 & 3  & 13 & -605750213184497 & 30878952 \\ \hline
 13 & 3  & 17 & -19809156065811865 & 59575360 &  13 & 5  & 3 & -3188621 & 2028 \\ \hline
 13 & 5  & 7 & -193778020789 & 357422 &  13 & 5  & 11 & -69045424287837 & 4188392 \\ \hline
 13 & 5  & 13 & -605750213184481 & 14891136 &  13 & 5  & 17 & -19809156065811849 & 89333920 \\ \hline
 13 & 7  & 3 & -3188597 & 1612 &  13 & 7  & 5 & -2441406201 & 35984 \\ \hline
 13 & 7  & 11 & -69045424287813 & 4845568 &  13 & 7  & 13 & -605750213184457 & 10482888 \\ \hline
 13 & 7  & 17 & -19809156065811825 & 104113152 &  13 & 11  & 3 & -3188525 & 1560 \\ \hline
 13 & 11  & 5 & -2441406129 & 35412 &  13 & 11  & 7 & -193778020693 & 183716 \\ \hline
 13 & 11  & 13 & -605750213184385 & 14582464 &  13 & 11  & 17 & -19809156065811753 & 70607680 \\ \hline
 13 & 13  & 3 & -3188477 & 1716 &  13 & 13  & 5 & -2441406081 & 32552 \\ \hline
 13 & 13  & 7 & -193778020645 & 228800 & 13 & 13  & 11 & -69045424287693 & 4533152\\ \hline
 13 & 13  & 17 & -19809156065811705 & 117589680 & 13 & 17  & 3 & -3188357 & 1560 \\ \hline
 13 & 17  & 5 & -2441405961 & 22464 & 13 & 17  & 7 & -193778020525 & 221000 \\ \hline
 13 & 17  & 11 & -69045424287573 & 5311488 &  13 & 17  & 13 & -605750213184217 & 13254488 \\ \hline
 15 & 3  & 5 & -61035156241 & 133620 &  15 & 3  & 7 & -9495123019877 & 3654000 \\ \hline
 15 & 3  & 11 & -8354496338831293 & 44413440 &  15 & 3  & 13 & -102371786028181505 & 389436480 \\ \hline
 15 & 3  & 17 & -5724846103019631577 & 1053896220 &  15 & 5  & 3 & -28697789 & 6330 \\ \hline
 15 & 5  & 7 & -9495123019861 & 2182740 &  15 & 5  & 11 & -8354496338831277 & 48860640 \\ \hline
 15 & 5  & 13 & -102371786028181489 & 231737760 &  15 & 5  & 17 & -5724846103019631561 & 2429285220 \\ \hline
 15 & 7  & 3 & -28697765 & 7140 &  15 & 7  & 5 & -61035156201 & 232380 \\ \hline
 15 & 7  & 11 & -8354496338831253 & 58597320 &  15 & 7  & 13 & -102371786028181465 & 184700760 \\ \hline
 15 & 7  & 17 & -5724846103019631537 & 1804868640 &  15 & 11  & 3 & -28697693 & 4740 \\ \hline
 15 & 11  & 5 & -61035156129 & 251460 &  15 & 11  & 7 & -9495123019765 & 1636920 \\ \hline
 15 & 11  & 13 & -102371786028181393 & 129010680 &  15 & 11  & 17 & -5724846103019631465 & 1339566720 \\ \hline
 15 & 13  & 3 & -28697645 & 3960 &  15 & 13  & 5 & -61035156081 & 177120 \\ \hline
 15 & 13  & 7 & -9495123019717 & 1139880 &  15 & 13  & 11 & -8354496338831133 & 43159500 \\ \hline
 15 & 13  & 17 & -5724846103019631417 & 970122240 &  15 & 17  & 3 & -28697525 & 7200 \\ \hline
 15 & 17  & 5 & -61035155961 & 294480 &  15 & 17  & 7 & -9495123019597 & 1418310 \\ \hline
 15 & 17  & 11 & -8354496338831013 & 39988560 &  15 & 17  & 13 & -102371786028181225 & 224843520 \\ \hline
 17 & 3  & 5 & -1525878906241 & 902632 &  17 & 3  & 7 & -465261027974405 & 21068168 \\ \hline
 17 & 3  & 11 & -1010894056998587533 & 340866116 &  17 & 3  & 13 & -17300831838762675857 & 3293665952 \\ \hline
 17 & 3  & 17 & -1654480523772673528345 & 16865341776 &  17 & 5  & 3 & -258280301 & 20672 \\ \hline
 17 & 5  & 7 & -465261027974389 & 13394096 &  17 & 5  & 11 & -1010894056998587517 & 483167744 \\ \hline
 17 & 5  & 13 & -17300831838762675841 & 2677442880 &  17 & 5  & 17 & -1654480523772673528329 & 45027534568 \\ \hline
 17 & 7  & 3 & -258280277 & 17068 &  17 & 7  & 5 & -1525878906201 & 887400 \\ \hline
 17 & 7  & 11 & -1010894056998587493 & 513329824 &  17 & 7  & 13 & -17300831838762675817 & 1808666448 \\ \hline
 17 & 7  & 17 & -1654480523772673528305  & 31484143616 &  17 & 11  & 3 & -258280205 & 16320 \\ \hline
 17 & 11  & 5 & -1525878906129 & 891072 &  17 & 11  & 7 & -465261027974293 & 7670808 \\ \hline
 17 & 11  & 13 & -17300831838762675745 & 2180205664 &  17 & 11  & 17 & -1654480523772673528233 & 18565719040 \\ \hline
 17 & 13  & 3 & -258280157 & 9248 &  17 & 13  & 5 & -1525878906081 & 1364692 \\ \hline
 17 & 13  & 7 & -465261027974245 & 18608472 &  17 & 13  & 11 & -1010894056998587373 & 805270688 \\ \hline
 17 & 13  & 17 & -1654480523772673528185 & 17835810304 &  17 & 17  & 3 & -258280037 & 9928 \\ \hline
 17 & 17  & 5 & -1525878905961 & 601936 &  17 & 17  & 7 & -465261027974125 & 8867200 \\ \hline
 17 & 17  & 11 & -1010894056998587253 & 526706580 &  17 & 17  & 13 & -17300831838762675577 & 2528985824 \\ \hline
 19 & 3  & 5 & -38146972656241 & 2652552 &  19 & 3  & 7 & -22797790370746277 & 153169260 \\ \hline
 19 & 3  & 11 & -122318180896829092573 & 5845471904 &  19 & 3  & 13 & -2923840580750892221345 & 54859889096 \\ \hline
 19 & 3  & 17 & -478144871370302649694297 & 246910222416 &  19 & 5  & 3 & -2324522909 & 39976 \\ \hline
 19 & 5  & 7 & -22797790370746261 & 174949568 &  19 & 5  & 11 & -122318180896829092557 & 5849478624 \\ \hline
 19 & 5  & 13 & -2923840580750892221329 & 28846572768 &  19 & 5  & 17 & -478144871370302649694281 & 548121230352 \\ \hline
 19 & 7  & 3 & -2324522885 & 50388 &  19 & 7  & 5 & -38146972656201 & 6622792 \\ \hline
 19 & 7  & 11 & -122318180896829092533 & 4729953024 &  19 & 7  & 13 & -2923840580750892221305 & 36032330968 \\ \hline
 19 & 7  & 17 & -478144871370302649694257 & 458154357332 &  19 & 11  & 3 & -2324522813 & 37240 \\ \hline
 19 & 11  & 5 & -38146972656129 & 4247868 &  19 & 11  & 7 & -22797790370746165 & 76028880 \\ \hline
 19 & 11  & 13 & -2923840580750892221233 & 26134377792 &  19 & 11  & 17 & -478144871370302649694185 & 389010811712 \\ \hline
 19 & 13  & 3 & -2324522765 & 40356 &  19 & 13  & 5 & -38146972656081 & 5545188 \\ \hline
 19 & 13  & 7 & -22797790370746117 & 55016248 &  19 & 13  & 11 & -122318180896829092413 & 6947197088 \\ \hline
 19 & 13  & 17 & -478144871370302649694137 & 397061284992 &  19 & 17  & 3 & -2324522645 & 36784 \\ \hline
 19 & 17  & 5 & -38146972655961 & 3776896 &  19 & 17  & 7 & -22797790370745997 & 77990972 \\ \hline
 19 & 17  & 11 & -122318180896829092293 & 8470575516 &  19 & 17  & 13 & -2923840580750892221065 & 35488586232 \\
  \end{longtable}}
\end{center}

\section{Concluding Remarks}
We begin by observing that in Table \ref{Table} there are some values of $p$ and $q$ (see $*$ mark) for which the class number of the corresponding imaginary quadratic fields are not divisible by a given odd integer $n\geq 3$. These are because of the failure of the assumption ``$3q^{n/3}\ne p+2$ when $3\mid n$". However the class number of $K_{19,7,3,2}$ is $12$ that satisfies the divisibility property even though this assumption does not hold. Thus this assumption is neither necessary nor sufficient. We have found only two pairs of values of $p$ and $q$ for which this assumption does not hold. Thus it may be possible to drop this assumption by adding some exceptions for the values of the pair $(p,q)$.  

In the light of the numerical evidence we are tempted to state the following conjecture:

\begin{conj}
Let $p$ and $q$ be two distinct odd primes. For each positive odd integer $n$ and  for each positive integer $m$ such that $m$ is not a $n$-root of any rational integer, there are infinitely many imaginary quadratic fields of the form 
$\mathbb{Q}(\sqrt{p^2-mq^n})$
whose class number is divisible by $n$.
\end{conj}
For $m=1, 4$, this conjecture is true (see \cite{CHKP18} and references therein). Further Corollary \ref{cor} shows that the conjecture is true for any odd integer $n$ when $m=2$.

%Finally, we demonstrate the prime parts of the class groups of $K_{p,q,n,2}$ in Table \ref{Table2}. 

 We now demonstrate the structures of class groups of $K_{p,q,n,2}$ for some values of $p,q$ and $n$. In Table \ref{Table2}, by $(h_1,h_2,\cdots,h_t)$ we mean the group $\mathbb{Z}_{h_1}\times \mathbb{Z}_{h_2}\times\cdots\times\mathbb{Z}_{h_t}$.   

\begin{center}
{\tiny
\begin{longtable}{|l l l l l l | l l l l l l||}
\caption{Structure of the class group of $K_{p,q,n,2}$.} \label{Table2} \\

\hline \multicolumn{1}{|c}{$p^2-2q^n$} & \multicolumn{1}{c}{Structure of $\mathcal{C}(K_{p,q,n,2})$} & \multicolumn{1}{c}{$2$-parts} & \multicolumn{1}{c}{$3$-parts} & \multicolumn{1}{c}{$5$-parts}& \multicolumn{1}{c|}{Remaining parts}\\ \hline 
\endfirsthead

\multicolumn{10}{c}%
{{\bfseries \tablename\ \thetable{} -- continued from previous page}} \\
\hline \multicolumn{1}{|c}{$p^2-2q^n$} & \multicolumn{1}{c}{Structure of $\mathcal{C}(K_{p,q,n,2})$} & \multicolumn{1}{c}{$2$-parts} & \multicolumn{1}{c}{$3$-parts} & \multicolumn{1}{c}{$5$-parts}& \multicolumn{1}{c|}{Remaining parts}\\ \hline 
\endhead
\hline \multicolumn{10}{|r|}{{Continued on next page}} \\ \hline
\endfoot
\hline 
\endlastfoot
$11^2-2\times 17^5$ & [20, 10, 2, 2] & (2,2,2,4)& -- & (5,5)& -- \\ \hline
$7^2-2\times 17^{13}$ & [1084512, 6, 2, 2, 2, 2] & (2,2,2,2,2,32) & (3) & -- & (11,13,79)\\ \hline
$13^2-2\times 11^{15}$ & [479550, 30, 3]& (2,2)&(3,3,3)&(5,25)&(23,139) \\ \hline
$13^2-2\times 17^{15}$ & [10105440, 12, 2, 2, 2]& (2,2,2,4,32)&(3,3)&(5)&(37,569)\\ \hline
$3^2-2\times 5^{21}$ & [565992, 6, 2, 2, 2]&(2,2,2,2,8)&(3,9)&--&(7,1123) \\ \hline
 $3^2-2\times 13^{21}$ & [7991268432, 6, 2, 2, 2, 2]&(2,2,2,2,2,16)&(3,3)&--&(7,23783537) \\ \hline
$17^2-2\times 11^{21}$ & [286454952, 12, 2, 2, 2, 2]&(2,2,2,2,4,8)&(3,9)&--&(7,568363) \\ \hline
$11^2-2\times 7^{25}$ & [292374800, 10, 2, 2, 2]& (2,2,2,2,16)&--&(5,25)&(101,7237)\\ \hline
$13^2-2\times 17^{25}$& [60345039225000, 6, 2, 2, 2]&(2,2,2,2,8)&(3,3)&(3125)&(41,59,332617) \\ \hline
$5^2-2\times 11^{27}$ &[381006210618, 6, 6, 2, 2, 2]**&(2,2,2,2,2,2)&(3,3,81)&--&(11,211,92119)\\ \hline
$5^2-2\times 13^{27}$ & [9392738579820, 6, 2, 2, 2, 2]&(2,2,2,2,2,4)&(3,27)&(5)&(59,9011,32717)\\ \hline
$5^2-2\times 17^{27}$ & [627358330621332, 6, 2, 2, 2, 2]&(2,2,2,2,2,4)&(3,81)&--&(183167,10571179)\\ \hline
$17^2-2\times 13^{27}$ & [8751451767912, 6, 2, 2, 2, 2]&(2,2,2,2,2,8)&(3,27)&--&(12251,67493)\\ \hline
$17^2-2\times 11^{29}$& [12592336322520, 2, 2, 2, 2, 2, 2]&(2,2,2,2,2,2,8)&(27)&(5)&(29,7411,54251)\\
 \end{longtable}}
\end{center}
\section*{Acknowledgements}
%The authors would like to thank the anonymous referee for valuable comments which helped correcting few errors. 
A. Hoque is supported by SERB N-PDF (PDF/2017/001758), Govt. of India.

\end{document}